\documentclass{amsart}
\usepackage{amsfonts,amssymb,amsmath,amsthm}
\usepackage{url}
\usepackage{enumerate}

\newtheorem{thm}{Theorem}[section]
\newtheorem{lem}[thm]{Lemma}
\newtheorem{prop}[thm]{Proposition}
\newtheorem{cor}[thm]{Corollary}
\theoremstyle{definition}

\newtheorem{rmk}[thm]{Remark}
\numberwithin{equation}{section}

\newcommand{\tf}{\widetilde{f}}
\newcommand{\ZZ}{\mathbb{Z}}
\newcommand{\CC}{\mathbb{C}}

\newcommand{\QQ}{\mathbb{Q}}
\DeclareMathOperator{\Ann}{Ann}
\DeclareMathOperator{\SL}{SL}
\DeclareMathOperator{\GL}{GL}
\newcommand*{\df}{\mathrel{\vcenter{\baselineskip0.5ex \lineskiplimit0pt
                     \hbox{\scriptsize.}\hbox{\scriptsize.}}} =}

\author[Cameron Franc \and Geoffrey Mason]{Cameron Franc \and Geoffrey Mason}
\address{Department of Mathematics\\
University of California, Santa Cruz}
\email{cfranc@ucsc.edu\\ gem@ucsc.edu}
\thanks{The second author is supported by the NSF}

\begin{document}

\title[Fourier coefficients of VVMFs of dimension $2$]{Fourier coefficients of vector-valued modular forms of dimension $2$}

\begin{abstract} We prove the following Theorem.\ Suppose that $F=(f_1, f_2)$ is a $2$-dimensional vector-valued modular form
on $\SL_2(\ZZ)$ whose component functions $f_1, f_2$ have \emph{rational} Fourier coefficients
with \emph{bounded denominators}.\ Then $f_1$ and $f_2$ are classical modular forms on a congruence subgroup of the modular group.\
MSC.
\end{abstract}
\maketitle

\section{Introduction}\label{S1}
Let $\Gamma= \SL_2(\ZZ)$ with 
$\rho:\Gamma\rightarrow \GL_n(\CC)$  a representation of $\Gamma$.\ For the purposes of the present paper, a \emph{vector-valued modular form of integral weight $k$ associated to $\rho$} is a column vector of
functions $F(\tau)=\ ^t(f_1(\tau), \hdots, f_n(\tau))$ holomorphic in the  upper half-plane and satisfying
\begin{eqnarray*}
F|_k\gamma(\tau) = \rho(\gamma)F(\tau)\ \quad (\gamma \in \SL_2(\ZZ)).
\end{eqnarray*}
Moreover, each component function is assumed to have a left-finite $q$-expansion  
 \begin{eqnarray*}
 f_i(\tau)=q^{m_i}\sum_{i=0}^{\infty} a_{ni}q^{n},
 \end{eqnarray*}
 where, as usual, $q=e^{2\pi i \tau}$ and $\tau$ is the coordinate on the upper half-plane.\ We are concerned here with vector-valued modular forms with the property
 that all of the Fourier coefficients $a_{ni}$ are \emph{rational numbers}.\ In this case it is known 
 \cite{MA} that the exponents $m_i$ are also rational.
 
\medskip
Suppose that $a(\tau)=q^h\sum_n a_nq^n$ is a $q$-expansion with 
 coefficients $a_{n}\in \QQ$.\ We say that $a(\tau)$ has  \emph{bounded denominators} if there is an integer $N$ such that
$Na_{n}\in \ZZ$ for all $n$.\ Otherwise, $a(\tau)$ has \emph{unbounded denominators}.\ If $F(\tau)$ is a vector-valued modular form whose components $f_i(\tau)$ have rational Fourier coefficients, we say that $F(\tau)$ has bounded (respectively unbounded) 
denominators if \emph{each}  $f_i$ (respectively, \emph{some} $f_i$) has bounded (respectively unbounded) denominators.\
The second author has conjectured (see \cite{M2} for the case of two-dimensional $\rho$) that the following are \emph{equivalent}:
\begin{eqnarray*}
&&\mbox{(a)\ $F(\tau)$ has rational Fourier coefficients with bounded denominators}, \\
&&\mbox{(b)\ Each $f_i(\tau)$ is a modular form on a congruence subgroup of $\Gamma$}.
\end{eqnarray*}

\medskip
The main result of the present paper is a proof of the conjecture for $2$-dimensional representations 
$\rho$.

\medskip
 The theory of $2$-dimensional vector-valued modular forms was developed in
\cite{M1}, \cite{M2}, and in particular the conjecture was proved in \cite{M2} for all but finitely
many $2$-dimensional $\rho$.\ We recall some of the ideas (under the assumption that $\rho$ is irreducible) since they will play a r\^{o}le
in the present paper.\ There is a
unique normalized nonzero holomorphic vector-valued modular form $F_0$
of \emph{least integral weight} $k_0$, and the components of $F_0$ constitute a fundamental
system of solutions of the linear differential equation
\begin{equation}
\label{diffeq}
  (D_{k_0+2}\circ D_{k_0} - k_1 E_4)f = 0.
\end{equation}
Here and below, we use the following notation\footnote{The normalization of $E_2$ used here differs from that in
\cite{M1}, \cite{M2}}:\ for an even integer $k \geq 2$, $E_k$ is the usual weight $k$ Eisenstein series with $q$-expansion
\[
  E_k(q) = 1 + \frac{2}{\zeta(1-k)}\sum_{n \geq 1} \sigma_{k-1}(n)q^n,
\]
and for an integer $k$ we have the weight $2$ operator
\[
  D_k = q\frac{d}{dq} -\frac{k}{12}E_2.
\]
Written in terms of $q$, (\ref{diffeq}) has a regular singular point
at $q=0$, the indicial roots are the exponents $m_1, m_2$, and
\begin{eqnarray}\label{kdef}
  k_0 = 6(m_1+m_2)-1\in \ZZ,\quad k_1 &=& \frac{36(m_1-m_2)^2-1}{144}.
\end{eqnarray}

\medskip
The recursive formula for the Fourier coefficients $a_n$
 of $f_1$ shows that they are rational when $f_1$ is suitably normalized.\ Moreover, it is shown in \cite{M2} that for almost all $\rho$, there is a prime $p$ (depending on $\rho$) such that the $p$-adic valuation
of $a_n$ is \emph{strictly decreasing} for $n\rightarrow \infty$.\ Thus $f_1$ has unbounded denominators
for such $\rho$.\ In the remaining exceptional cases (approximately $300$ isomorphism classes of $\rho$) this method will fail because, as numerical computations show,  there is no such prime $p$.\ Thus another device is needed to achieve unbounded denominators in these cases.\ (The exceptions
include $54$ classes of \emph{modular} $\rho$ for which the components are modular forms on a
congruence subgroup, and for these cases one of course  has bounded denominators.)\ Further techniques are used to show that if
$F_0$ has unbounded denominators then every nonzero $F$ with rational Fourier coefficients has the same property.

\medskip
In the present paper we will show that for all choices of $\rho$, exceptional or not, the component functions $f_1, f_2$
can be described using Gauss's hypergeometric function $F(a, b; c; j^{-1})$ evaluated at the inverse of
the absolute modular invariant $j$ (cf.\ Proposition \ref{prop1} below).\ This will allow us to
show (in the nonmodular cases) that \emph{infinitely many primes} occur in the denominators of the Fourier coefficients, and in particular that denominators are unbounded.\ 
 In fact, more precise arithmetic information is available in this situation, as we will explain in due course. 
 
 \medskip
 Hypergeometric series appear in the work of Bantay and Gannon \cite{BG} on vector-valued modular forms and
 the `fundamental matrix'. In hindsight we observed that the technique of this paper is strongly suggested by the work of Bantay and Gannon, but we in fact drew our inspiration from an earlier paper of Kaneko and Zagier \cite{KZ} on supersingular $j$-invariants.\
 Kaneko and Zagier considered a special case of (\ref{diffeq}) for which one of the solutions
 is a modular form (the corresponding $\rho$ is \emph{indecomposable}),
 and they use a change of local variable to reexpress the modular form in terms of hypergeometric series. Other papers, for example \cite{KK} and \cite{T}, have also studied connections between modular linear differential equations and hypergeometric differential equations.

\medskip
We will prove the following results.
\begin{thm}
\label{thm1:main}
Let $m_1, m_2$ be rational numbers such that $m_1-m_2=P/Q$,
gcd$(P, Q)=1$, and $Q\geq 2$, and let $k_0, k_1$ be as in (\ref{kdef}).\ Then (\ref{diffeq}) has two linearly independent solutions $f_1, f_2$  with rational $q$-expansions, and exactly one of the following is true:
\begin{enumerate}
\item at least one of $f_1$ or $f_2$ has unbounded denominators,
\item $Q\leq 5$.
\end{enumerate}
\end{thm}

\begin{thm}
\label{thm2:main}
  Let $F$ denote \emph{any} $2$-dimensional vector valued modular form whose components have rational Fourier coefficients.\ Then exactly one of the following is true:
\begin{enumerate}
\item at least one of the components of $F$ has unbounded denominators;
\item both components of $F$ are modular forms on a congruence subgroup.
\end{enumerate}
\end{thm}

\begin{rmk}\ As we have explained, the components of the vector-valued modular form $F_0$
satisfy the assumptions (hence also the conclusions) of Theorem \ref{thm1:main}.\ However, most
choices of $m_1$ and $m_2$ do \emph{not} correspond to any $\rho$.\ Theorem \ref{thm1:main}
teaches us that
the origin of the unbounded denominator phenomenon is not so much the representation $\rho$,
but rather the differential equation (\ref{diffeq}) to which it is associated.
\end{rmk}

We shall actually prove a much more precise result than is stated in Theorem \ref{thm1:main}.\ If $Q\geq 6$
we will show that at least one of the following statements holds:
\begin{eqnarray}\label{ap}
&&\mbox{a) $m_1>m_2$ and for every prime $p$ in the arithmetic progression $Qn+P$,}\notag\\
&&\mbox{one of the Fourier coefficients of $f_1$ has $p$-adic valuation $-1$;} \\
&&\mbox{b) $m_1<m_2$ and for every prime $p$ in the arithmetic progression $Qn-P$,}\notag\\
&&\mbox{one of the Fourier coefficients of $f_2$ has $p$-adic valuation $-1$;} \notag
\end{eqnarray}

\section{A modular change of variable}
In this Section, $m_1, m_2, k_0, k_1, P, Q$ are assumed to satisfy the conditions stated in
Theorem \ref{thm1:main}.\ 
Let $\eta(q)$ denote Dedekind's eta function
\[
 \eta(q) = q^{1/24}\prod_{n = 1}^\infty(1 - q^n).
\]
\begin{lem}
\label{l:etaderivative}
One has $D_1(\eta^2) = 0$.
\end{lem}
\begin{proof}
Let $\Delta = \eta^{24}$ and recall the well-known identity $q\frac{d\Delta}{dq} = E_2\Delta$. This is equivalent with $D_{12}(\Delta) = 0$. The identity $D_1(\eta^2) = 0$ follows from this by application of the Leibniz rule.
\end{proof}
Let $f$ denote a solution of (\ref{diffeq}). In order to study the $q$-expansion of $f$ we introduce the change of variable $\tf \df f\eta^{-2k_0}$. Lemma \ref{l:etaderivative} shows that $f$ is a solution of (\ref{diffeq}) if and only if $\tf$ satisfies
\begin{equation}
\label{diffeq2}
D_2\circ D_0 \tf - k_1E_4\tf = 0.
\end{equation}
We will show that this is a hypergeometric equation when expressed in terms of the local parameter $j^{-1}$ at infinity. To begin, let $\theta = q\frac{d}{dq}$. The differential equation (\ref{diffeq2}) is equivalent with the equation
\begin{equation}
\label{diffeq3}
\theta^2(\tf) -\frac{1}{6}E_2\theta(\tf) - k_1E_4 \tf = 0.
\end{equation}
We will reexpress equation (\ref{diffeq3}) in terms of $J : = j/1728$. Note that
\begin{align*}
  \frac{j}{j - 1728} &= \frac{E_4^3}{E_6^2},\\
  \frac{d\tf}{dj} &= -\frac{E_4}{jE_6}\theta(\tf),\\
  \frac{d^2\tf}{dj^2} &= \left(\frac{E_4}{jE_6}\right)^2\left(\theta^2(\tf)-\frac{E_2}{6}\theta(\tf)\right) - \left(\frac{7j-4\cdot 1728}{6j(j-1728)}\right)\frac{d\tf}{dj}.
\end{align*}
Then  (\ref{diffeq3}) becomes
\[
j(j-1728)\frac{d^2\tf}{dj^2} +\frac{7j-4.1728}{6}\frac{d\tf}{dj}  - k_1 \tf = 0,
\]
which is equivalent to the Gauss normal form
\begin{equation}
\label{eq:hggen}
J(1-J)\frac{d^2\tf}{dJ^2} +\left(\frac{4-7J}{6}\right)\frac{d\tf}{dJ} + k_1 \tf = 0.
\end{equation}
(Here, and below, we write $J=j/1728$.)\
The general Gauss normal form is expressed in terms of parameters $a$, $b$ and $c$ as
\[
  J(1-J)\frac{d^2\tf}{dJ^2} + (c -(a + b + 1)J)\frac{d\tf}{dJ} - ab\tf = 0.
\]
This corresponds to (\ref{eq:hggen}) when
\begin{eqnarray}\label{abcdef}
  a = \frac{1}{12} + \left(\frac{m_1-m_2}{2}\right),\quad\ b= \frac{1}{12} - \left(\frac{m_1-m_2}{2}\right),\quad
   c= \frac{2}{3}.
\end{eqnarray}
Observe that $a-b = m_1 - m_2$ is \emph{not} an integer (because $Q\geq 2$).\ Thus (\ref{eq:hggen}) has two independent solutions at $J = \infty$ given by
\[
J^{-a}F(a,1+a-c;1+a-b;J^{-1}),\quad J^{-b}F(b, 1 +b-c; 1+b-a;J^{-1}),
\]
where $F(a,b;c;z)$ is Gauss's hypergeometric function
\begin{eqnarray}\label{hgs}
  F(a,b;c;z) \df 1+\sum_{n \geq 1}\frac{(a)_n(b)_n}{(c)_n(1)_n} z^n.
\end{eqnarray}
This proves the following.
\begin{prop}
\label{prop1}
Assume that $m_1, m_2, k_0, k_1, P, Q$ are as in the statement of Theorem \ref{thm1:main}.\
Then (\ref{diffeq}) has two linearly independent solutions given by the series
\begin{eqnarray*}
f_1&=&  \eta^{2k_0}J^{-a}F(a,1+a-c;1+a-b;J^{-1}),\\
 f_2&=&\eta^{2k_0}J^{-b}F(b, 1 +b-c; 1+b-a;J^{-1}),
\end{eqnarray*}
where $a, b, c$ are as in (\ref{abcdef}). $\hfill \Box$
\end{prop}
\begin{rmk}
  The lowest terms in the $q$-expansions of $f_1$ and $f_2$ have exponents $m_2$ and $m_1$ respectively. Also, the $q$-expansions of $f_1$ and $f_2$ are rational save for possibly up to a fractional power of $12$ arising from the terms $J^{-a}$ and $J^{-b}$.
\end{rmk}

Let us consider  the first of these series, in particular the coefficients
given by the corresponding Pochhammer symbols occurring in (\ref{hgs}).\ 
For $n \geq 1$, the coefficients in question are
\begin{align}\label{phrat1}
C_n:=  \frac{(a)_n(1+a-c)_n}{(1+a-b)_n(1)_n}
= (144Q)^{-n}\prod_{k = 0}^{n-1} \frac{(12Qk +Q+ 6P)(12Qk + 5Q + 6P)}{(Qk+Q + P)(k+1)}.
\end{align}

\medskip
Suppose that $p=Qn+P$ is a \emph{prime} that divides
$(12Qk+Q+6P)(12Qk+5Q+6P)$ for some $k$ in the range $0\leq k \leq n-1\ (n\geq 1)$.\
If $p$ divides the first factor then it divides $12Qk+Q+6P-6(Qn+P)=Q(12k+1-6n)$, 
so that
\begin{eqnarray*}
Qn+P|12k+1-6n \leq 6n-11 \Rightarrow n(Q-6)\leq -(P+11).
\end{eqnarray*}
Similarly, if $p$ divides the second
factor then
\begin{eqnarray*}
Qn+P|12k+5-6n\leq 6n-7\Rightarrow n(Q-6)\leq -(P+7).
\end{eqnarray*}
In particular, if $P$ is positive (i.e.\ $m_1>m_2$) and $Q\geq 6$ then neither of these conditions can hold, so 
 $p$ cannot divide the numerator of $C_n$.\ It is then evident that the $p$-adic valuation of $C_n$ is exactly $-1$.
For the second hypergeometric series we consider the coefficients 
 \begin{align}\label{phrat2}
C'_n:=  \frac{(a)_n(1+b-c)_n}{(1+b-a)_n(1)_n}
= (144Q)^{-n}\prod_{k = 0}^{n-1} \frac{(12Qk +Q- 6P)(12Qk + 5Q - 6P)}{(Qk+Q - P)(k+1)}.
\end{align}
We easily find results similar to those obtained in the first case, but now for primes $Qn-P$ and $P<0$.\
Combining these results yields a proof of the following lemma.
 \begin{lem}\label{lemQ>5} Suppose that $Q\geq 6$.\ Then one of the following holds: 
\begin{enumerate}
\item[(a)] $m_1>m_2$ and every prime $p=Qn+P$ is such that the $p$-adic valuation of $C_n$ is $-1$;
\item[(b)] $m_2>m_1$ and every prime $p=Qn-P$ is such that the $p$-adic valuation
 of $C'_n$ is $-1$. $\hfill \Box$
\end{enumerate}
\end{lem}

\begin{prop}\label{prop2}Let $m_1, m_2, k_0, k_1, P, Q$ be as in the statement of Theorem \ref{thm1:main}, and
assume further that $Q\geq 6$.\ Let $f_1, f_2$ be the two $q$-expansions in Proposition \ref{prop1}.\ Then one of the following holds.

\begin{enumerate}
\item[(a)] $m_1>m_2$ and for every prime $p$ in the arithmetic progression $Qn+P$ there is at least one Fourier coefficient of $f_1$
 that has $p$-adic valuation  $-1$;
\item[(b)] $m_2>m_1$ and for every prime $p$ in the arithmetic progression $Qn-P$ there is at least one Fourier coefficient of $f_2$
 that has $p$-adic valuation  $-1$.
\end{enumerate}
 \end{prop}
 \begin{proof} Suppose that $m_1>m_2$, and fix a prime $p=Qn+P\ (n\geq 1)$.\ We have
 \begin{eqnarray}\label{jform}
\eta^{-2k_0}J^af_1 = 1728^{-a}\left(1+\sum_{n\geq 1} (12)^{3n}C_nj^{-n}\right) =:1728^{-a} \left(1+\sum_{n\geq 1} c_nq^n\right),
\end{eqnarray}
and by part (a) of Lemma \ref{lemQ>5} the $p$-adic valuation of $C_n$ is $-1$.\ It follows that the $p$-adic valuation of $c_n$ is $-1$ while that of $c_m$
is \emph{nonnegative} for $1\leq m\leq n-1$.

\medskip
If all  Fourier coefficients of $f_1$ have nonnegative $p$-adic valuation, the same is true of $\eta^{-2k_0}j^af_1$.\ This is
because the $\eta$-power has integral coefficients, while the only primes occurring in the denominators of coefficients of
$j^a = 1728^aJ^{a}$ divide $12Q$ (cf.\ (\ref{abcdef})), which is coprime to $p$.\ This contradicts the previous paragraph, and thus shows that some Fourier coefficient
of $f_1$ has negative $p$-adic valuation.\ A similar argument shows that the first such coefficient has $p$-adic valuation exactly $-1$,
because that is true of the coefficients $c_m$.\ This completes the proof of the Proposition in case $m_1>m_2$.\
The proof in the case $m_2>m_1$ is completely parallel. 
\end{proof}

Notice that both (\ref{ap}) and Theorem \ref{thm1:main} are consequences of Proposition \ref{prop2}.

\section{Vector-valued modular forms}
This Section is devoted to the proof of Theorem \ref{thm2:main}.\ We first develop some general results concerning vector-valued modular forms
whose components have rational Fourier coefficients.\ We use the following additional notation:
\begin{enumerate}
\item[---] $\frak{M} $ is the algebra of (classical) holomorphic modular forms on $\Gamma$.
\item[---] $\frak{M}_{\QQ}$ is the $\QQ$-algebra of holomorphic modular forms with Fourier coefficients in $\QQ$.
\item[---] $\frak{p} = \frak{M}_{\QQ}\Delta$ is the principal ideal of $\frak{M}_{\QQ}$ generated by the discriminant $\Delta$.
\item[---]$\rho:\Gamma \rightarrow \GL_n(\CC)$ is an $n$-dimensional  representation of $\Gamma$ such that $\rho(T)$ is (similar to) a unitary matrix.
\item[---] $\mathcal{H}(\rho)$ is  the $\ZZ$-graded space of holomorphic vector-valued modular forms associated to $\rho$; it is a free $\frak{M}$-module of rank $n$ (\cite{MM}).
\item[---] $\mathcal{H}(\rho)_{\QQ}$ is the space of vector-valued modular forms in $\mathcal{H}(\rho)$,
 all of whose component functions have Fourier coefficients in $\QQ$; it is a module over $\frak{M}_{\QQ}$.
\item[---] $\frak{R}$ is the (noncommutative) polynomial ring $\frak{M}[d]$ such that
 $df-fd=D(f)\ (f\in\frak{M});$ $\mathcal{H}(\rho)$ is a left $\frak{R}$-module where 
 $f\in \frak{M}$ acts as multiplication by $f$ and $d$ acts as $D$.\ Similarly, 
 $\frak{R}_{\QQ} = \frak{M}_{\QQ}[d]$ and $\mathcal{H}(\rho)_{\QQ}$ is
 a left $\frak{R}_{\QQ}$-module (\cite{M1}).
\end{enumerate}

\bigskip
 The next result is  technical, but very useful.
 \begin{prop}\label{propRQmod} Assume that $\rho$ is \emph{irreducible}.\ If $I \subseteq \mathcal{H}(\rho)_{\QQ}$ is a 
 \emph{nonzero} $\frak{R}_{\QQ}$-submodule, then there is an integer $r$ such that $\frak{p}^r\mathcal{H}(\rho)_{\QQ} \subseteq I$. In other words, $\frak{p}^r$
 \emph{annihilates} $\mathcal{H}(\rho)_{\QQ}/I$.
 \end{prop}
\begin{proof} Let $A = \Ann_{\frak{M}_{\QQ}}(\mathcal{H}(\rho)_{\QQ}/I)$. It is an ideal in $\frak{M}_{\QQ}$, and we have to show that
$\frak{p}^r \subseteq A$ for some $r$.

\medskip
Choose any nonzero $F \in I$, and consider the vector-valued modular forms $D^jF\ (0 \leq j \leq n-1)$. If they are linearly dependent over
$\frak{M}_{\QQ}$ then  the components of $F$ satisfy a modular linear differential equation of order $\leq n-1$, and hence are linearly dependent. Because $\rho$ is irreducible this is not possible. It follows that the $D^jF$ span an $\frak{M}_{\QQ}$-submodule of $\mathcal{H}(\rho)_{\QQ}$ of (maximal) rank $n$. So  $I$ is also a submodule of maximal rank $n$.

\medskip
We claim that $\mathcal{H}(\rho)_{\QQ}/I$ is a \emph{torsion} $\frak{M}_{\QQ}$-module.
If not, we can find $G \in \mathcal{H}(\rho)_{\QQ}/I$ such that the annhilator of $G$ in $\mathcal{M}_{\QQ}$ reduces to $0$. Then the submodule generated by $G$ is a rank $1$ free module, call it $J$, and there is a short exact sequence 
\begin{eqnarray*}
0 \rightarrow I \rightarrow K \rightarrow J \rightarrow 0
\end{eqnarray*}
of $\mathcal{M}_{\QQ}$-modules. Because $J$ is free the sequence splits and we obtain $K \cong I \oplus J$ which is free of rank $n+1$. This is not possible because $\mathcal{H}(\rho)_{\QQ}$ has rank $n$, whence no submodule has rank  greater than $n$.\ This proves the claim.

\medskip
Because $\mathcal{H}(\rho)_{\QQ}/I$ is a torsion module and $\frak{M}_{\QQ}$ a domain, it follows that $A$ is \emph{nonzero}. It is also easy to see that $A$ is a graded ideal of $\frak{M}_{\QQ}$.\ Furthermore, because $I$ is a left $\frak{R}_{\QQ}$-submodule
and $D$ is a derivation, an easy calculation shows that $d$ (aka $D$) leaves
$A$ invariant.\ Hence, $A$ is a nonzero, graded, left $\frak{R}_{\QQ}$-submodule
of $\frak{M}_{\QQ}$.

\medskip
 In Lemmas 2.6 and 2.7 of
\cite{MM} it was  proved that a nonzero graded $\frak{R}$-submodule of $\frak{M}$ contains 
$\frak{M}\Delta^r$ for some $r$.\ A check of the proof shows that it still works if the base field $\CC$ is replaced by $\QQ$ and $\frak{M}\Delta$ is replaced by
$\frak{p}$. So $\frak{p}^r \subseteq A$ for some $r$, as required.\  This completes the proof of the Proposition. 
\end{proof}

\begin{cor}\label{c2orbdd} Assume that $\rho$ is irreducible.\ Then the following are equivalent:
\begin{enumerate}
\item[(a)] $\mathcal{H}(\rho)_{\QQ}$ contains at least one nonzero vector-valued modular form
with bounded denominators;
\item[(b)]\emph{Every} nonzero element in $\mathcal{H}(\rho)_{\QQ}$ has bounded denominators.
\end{enumerate}
\end{cor}
\begin{proof} The  set of vector-valued modular forms in $\mathcal{H}(\rho)_{\QQ}$ with
bounded denominators is an $\frak{R}_{\QQ}$-submodule of
$\mathcal{H}(\rho)_{\QQ}$, call it $I$.\ If $I \not= 0$ then Proposition \ref{propRQmod} applies.\
 It tells us that
$\Delta^r\mathcal{H}(\rho)_{\QQ}\subseteq I$ for some $r$.\ Thus for any 
$F \in \mathcal{H}(\rho)_{\QQ}$ we find that $\Delta^rF$ has bounded denominators, whence $F = \Delta^{-r}\Delta^rF$ does too.\ The Corollary follows. 
\end{proof}

We turn to the proof of Theorem \ref{thm2:main} and suppose that  $F=\ ^t(g_1, g_2)\in \mathcal{H}(\rho)_{\QQ}$ has weight $k$ and  that $g_1$ and $g_2$ are \emph{not} both modular forms on a congruence subgroup of $\Gamma$.\ We have to show that 
$F$ has unbounded denominators.\ If $\rho$ is irreducible, it suffices by Corollary \ref{c2orbdd} to find \emph{one}
vector-valued modular form in $\mathcal{H}(\rho)_{\QQ}$ with unbounded denominators.\ As explained in the Introduction,
$\mathcal{H}(\rho)$ has  a unique (normalized) nonzero vector-valued modular form
$F_0=\ ^t(f_1, f_2)$ of minimal weight, and $F_0$ has rational Fourier coefficients.\ Let
$f_1=q^{m_1}+..., f_2=q^{m_2}+...$ with notation as in Theorem \ref{thm1:main}; in particular,
$m_1-m_2=P/Q$ with gcd$(P, Q)=1, Q\geq 1$.
$f_1$ and $f_2$ form a fundamental system of solutions of the DE (\ref{diffeq}).\ If $Q=1$ then
$m_1=m_2$ is an integer, and in this case $\rho(T)$ is a scalar (cf.\ \cite{M1}).\ Because $\rho$ is irreducible this is not the case,
so that $Q\geq 2$.

\medskip
 Now we can apply Theorem \ref{thm1:main} to see that either $Q\leq 5$, or else one of $f_1, f_2$ has unbounded denominators.\ In the second case we are done.\
We show that $Q\leq 5$ leads to a contradiction.\ Indeed, in this case Proposition 3.2 of \cite{M2} tells us
that ker$\rho$ is a congruence subgroup of
$\Gamma$.\ But then all components of all vector-valued modular forms in $\mathcal{H}(\rho)$ are
modular forms on the same congruence subgroup, and this contradicts the existence of $F$.
This completes the proof of Theorem \ref{thm2:main} in the case that $\rho$ is irreducible.

\medskip
Now suppose that $\rho$ is \emph{not} irreducible.\ Then we may, and shall, assume that it is
upper triangular,
\begin{eqnarray}\label{indec1}
\rho(\gamma) = \left(\begin{array}{cc}\alpha(\gamma) & \beta(\gamma) \\0 & \delta(\gamma)\end{array}\right)\ \  (\gamma \in \Gamma).
\end{eqnarray}
Note that in this situation, the analog of Corollary
\ref{c2orbdd} is \emph{false}.\ Indeed,  $F=\ ^t(f_1, 0)\in \mathcal{H}(\rho)_{\QQ}$ 
has weight $k$ and bounded denominators whenever $f_1$ is a modular form of weight 
$k$ on $\Gamma$ with character $\alpha$ and rational Fourier coefficients.

\medskip
Let $I\subseteq \mathcal{H}(\rho)_{\QQ}$ be the set of vector-valued modular forms with bounded denominators.\ It is an $\frak{R}_{\QQ}$-submodule, and by our preceding remarks it
contains the space of functions $\frak{M}' = \{^t(f_1, 0)\}$ described above.\ Suppose, by way of contradiction, that $F\in I$.\ Since $F\not\in \frak{M}'$ it follows that $I$ has rank at least $2$ considered as $\frak{M}_{\QQ}$-module.\ At this point, we can apply the proofs of 
Proposition \ref{propRQmod} and Corollary \ref{c2orbdd} word-for-word to see that every
nonzero vector-valued modular form in $\mathcal{H}(\rho)_{\QQ}$ has bounded denominators.

\medskip
We now apply results about indecomposable $2$-dimensional $\rho$ and their associated vector-valued modular forms obtained in
\cite{MM}, Section 4.\ By Lemma 4.3 (loc.\ cit.) we always have $Q=6$ in this case.\ Let 
$F_0=\ ^t(f_1, f_2)\in \mathcal{H}(\rho)_{\QQ}$ be as before, i.e.\ a nonzero vector-valued modular form of least weight
$k_0$.\ As in \cite{MM} we distinguish two cases, according to whether $DF_0=0$ or not.\
 If this does \emph{not} hold then $f_1, f_2$ are, once again, a fundamental system of solutions
of the DE (\ref{diffeq}) and we can apply Theorem 1 immediately to conclude that $F_0$ has unbounded denominators, which is the desired contradiction in this case.\ Note that we may always choose $\rho(T)$ diagonal, in which case
we have in this case ( eqn.\ (27) of \cite{MM}) 
\begin{eqnarray}\label{rhochoice}
&&\rho(T) = \left(\begin{array}{cc}e^{2\pi i m_1} & 0\\0 & e^{2\pi i m_2} \end{array}\right),\\ 
&&0\leq m_2<m_1<1, m_1-m_2 = 1/6\ \mbox{or}\ 5/6, 12m_i\in \ZZ. \notag
\end{eqnarray}

\medskip
Finally, suppose that $DF_0=0$.\ This holds for those indecomposable $\rho'$ for which
\begin{eqnarray*}
&&\rho'(T) = \left(\begin{array}{cc}e^{2\pi i m_1} & 0\\0 & e^{2\pi i m_2} \end{array}\right),\\
&&0\leq m_1<m_2<1, m_2-m_1 = 1/6\ \mbox{or}\ 5/6, 12m_i\in\ZZ.
\end{eqnarray*}
In this case, consider the tensor product $\rho'':=\chi\otimes \rho'$ where $\chi:\Gamma \rightarrow \CC^*$
is the character of $\Gamma$ satisfying $\chi(T)=e^{-2\pi im_2}.$\ Because the isomorphism classes of
indecomposable $2$-dimensional $\rho$ are determined by $\rho(T)$ (\cite{MM}, Lemma 4.3), it
follows that $\rho''$ is equivalent to that $\rho$ in (\ref{rhochoice}) for which $m_2=0$.\ Now we have already proved that $\mathcal{H}(\rho)_{\QQ}$ contains some vector-valued modular form with unbounded denominator, so $\mathcal{H}(\rho'')_{\QQ}$ also contains such a vector-valued modular form, say  $G$.\ Then $\rho' = \chi^{-1}\otimes \rho''$, and if
$m_2=b/12\ (b\in \ZZ)$ then $\eta^{2b}G\in \mathcal{H}(\rho')_{\QQ}$ also has unbounded denominators.\ This final contradiction completes the proof of 
Theorem \ref{thm2:main}. $\hfill \Box$

\section{Final remarks}

Prior to the writing of this paper, the authors were quite mystified by the nature of the denominators of the Fourier coefficients of the coordinate functions of vector valued modular forms that are not themselves modular forms.\ These denominators tend to be nearly squarefree and, with a few exceptional divisors, are divisble only by primes in at most two arithmetic progressions\footnote{Chris Marks has made similar observations.}.\ For example, when one takes $m_1 = 3/10$ and $m_2 = 2/10$, there is a unique irreducible representation
\[
  \rho \colon \SL_2(\ZZ) \to \GL_2(\CC)
\]
such that
\[
  \rho\left(\begin{matrix}
          1&1\\
          0&1\end{matrix}\right) = \left(\begin{matrix}e^{3\pi i /5}&0\\
  0&e^{2\pi i/5}\end{matrix}\right).
\]
Let $F_0$ be a nonzero vector valued modular form for $\rho$ of lowest weight.\ Then $F_0$ may be rescaled to have rational Fourier coefficients such that the $1,000$th Fourier coefficient of the first component function of $F_0$ has denominator equal to $3^2\cdot 13$ times the product of every prime in the arithmetic progression $10n + 9$ in the range $0$ through $10,000$.\ The denominator of the $1,001$st Fourier coefficient of this same $q$-expansion is $3$ times the product of every prime in the arithmetic progression $10n+9$ in the range $0$ through $10,009$.\ For the $1,002$nd coefficient, however, the denominator is $13$ times the product of all primes in the progression $10n+9$ in the same range, except that $919$ is omitted.\ While the connection with hypergeometric series does help to explain the origin of the arithmetic progressions, it is still somewhat mysterious why the denominators of these series tend to be very nearly squarefree products of all primes in one or two arithmetic progressions.\ One might naively expect far more cancellation to appear than seems to be the case.\ A deeper study of the arithmetic properties of these coefficients might prove interesting.

\medskip
The authors hope that the techniques of this paper will lead to progress in understanding vector-valued modular forms associated to higher dimensional representations of the modular group.\ Initial computations suggest that they too are connected with higher order analogues of the hypergeometric differential equation.

\end{document}